\documentclass[12pt]{amsart}

\usepackage[all]{xypic}
\usepackage{tikz}
\usetikzlibrary{arrows} 
\usetikzlibrary{decorations.markings}
\usepackage{graphicx}
\usepackage{bm}
\usepackage{float}
\usepackage{epsf}
\usepackage{verbatim} 
\usepackage{amsmath}
\usepackage{amsfonts}
\usepackage{amssymb}
\usepackage{mathrsfs}
\usepackage{amsthm}
\usepackage{newlfont}
\usepackage{enumitem}
\usepackage[new]{old-arrows}
\usepackage{booktabs}
\usepackage{enumitem}
\usepackage{makecell}
\usepackage[multiple]{footmisc}
\usepackage[answerdelayed,lastexercise]{exercise}
\usepackage[hidelinks]{hyperref}
\usepackage{fnpct}
\usepackage{colonequals} 
\usepackage{stmaryrd}

\theoremstyle{plain}
\newtheorem{introtheorem}{Theorem}

\newtheorem{thm}{Theorem}[section]
\newtheorem{prop}[thm]{Proposition}
\newtheorem{lem}[thm]{Lemma}
\newtheorem{cor}[thm]{Corollary}

\theoremstyle{definition}
\newtheorem{defn}[thm]{Definition}

\newtheorem{example}[thm]{Example}

\theoremstyle{remark}
\newtheorem{rem}[thm]{Remark}

\makeatletter
\@addtoreset{footnote}{section}
\makeatother

\usepackage{xcolor}
\hypersetup{
	colorlinks,
	linkcolor={red!50!black},
	citecolor={blue!50!black},
	urlcolor={blue!80!black}
}

\AtBeginDocument{%
	\def\MR#1{}
}

\DeclareMathOperator{\Hom}{Hom}
\DeclareMathOperator{\End}{End}

\DeclareMathOperator{\rank}{rank}

\DeclareMathOperator{\chr}{char}

\DeclareMathOperator{\ord}{ord}

\DeclareMathOperator{\Ann}{Ann}

\DeclareMathOperator{\Cent}{Cent}

\newcommand{\fl}{\mathfrak{l}}

\newcommand{\fp}{\mathfrak{p}}

\newcommand{\cE}{\mathcal{E}}

\newcommand{\cM}{\mathcal{M}}

\newcommand{\cO}{\mathcal{O}}

\newcommand{\F}{\mathbb{F}}
\newcommand{\gm}{\mathbb{G}}
\renewcommand{\H}{\mathbb{H}}

\newcommand{\tF}{\widetilde{F}}

\newcommand{\tfp}{\widetilde{\fp}}

\newcommand{\To}{\longrightarrow}

\newcommand{\La}{\Lambda}
\newcommand{\la}{\lambda}

\newcommand{\twist}[1]{#1\!\left\{\tau\right\}}

\newcommand{\ls}[2]{#1\!\left(\mkern-4mu\left(#2\right)\mkern-4mu\right)}
\newcommand{\dvr}[2]{#1\!\left\llbracket #2\right\rrbracket}

\begin{document}
\title{Isomorphism classes of Drinfeld modules over finite fields}

\author{Valentijn Karemaker}
\address{Mathematical Institute, Utrecht University, Utrecht, The Netherlands}
\email{V.Z.Karemaker@uu.nl}

\author{Jeffrey Katen}
\address{Department of Mathematics, Pennsylvania State University, University Park, Pennsylvania, United States of America}
\email{jmk1097@psu.edu}

\author{Mihran Papikian}
\address{Department of Mathematics, Pennsylvania State University, University Park, Pennsylvania, United States of America}
\email{papikian@psu.edu}


\begin{abstract}
We study isogeny classes of Drinfeld $A$-modules over finite fields $k$ with commutative endomorphism algebra $D$, in order to describe the isomorphism classes in a fixed isogeny class. We study when the minimal order $A[\pi]$ of $D$ occurs as an endomorphism ring by proving when it is locally maximal at~$\pi$, and show that this happens if and only if the isogeny class is ordinary or $k$ is the prime field. We then describe how the monoid of fractional ideals of the endomorphism ring $\mathcal{E}$ of a Drinfeld module $\phi$ up to $D$-linear equivalence acts on the isomorphism classes in the isogeny class of $\phi$, in the spirit of Hayes. We show that the action is free when restricted to kernel ideals, of which we give three equivalent definitions, and determine when the action is transitive. In particular, the action is free and transitive on the isomorphism classes in an isogeny class which is either ordinary or defined over the prime field, yielding a complete and explicit description in these cases.
\end{abstract}

\maketitle


\section{Introduction}

Let $\F_q$ be a finite field with $q$ elements. Let $A=\F_q[T]$ be the ring of polynomials in indeterminate $T$ 
with coefficients in $\F_q$, and let $F=\F_q(T)$ be the fraction field of~$A$. 
Given a nonzero prime ideal $\fp$ of $A$,  
we denote $\F_\fp=A/\fp$.
Let $k\cong \F_{q^n}$ be a finite extension of $\F_\fp$. We consider $k$ as an $A$-field via 
$\gamma\colon A \to A/\fp \hookrightarrow k$. 

Let $\tau$ be the Frobenius automorphism of $k$ relative to $\F_q$, that is, the map $\alpha\mapsto \alpha^q$. 
Let $\twist{k}$ be the noncommutative ring of polynomials in $\tau$ with coefficients in $k$ and commutation rule 
$\tau \alpha=\alpha^q\tau$, $\alpha\in k$. A Drinfeld module of rank $r\geq 1$ over $k$ is a ring 
homomorphism $\phi\colon A\to \twist{k}$, $a\mapsto \phi_a$, such that 
\[
\phi_a=\gamma(a)+g_1(a)\tau+\cdots+g_n(a)\tau^n, \quad n=r\deg_T(a). 
\]
(Note that $\phi$ is uniquely determined by $\phi_T$.) 
An isogeny $u\colon \phi\to \psi$ between two Drinfeld modules over $k$ is a nonzero element $u\in \twist{k}$ 
such that $u\phi_a=\psi_a u$ for all $a\in A$ (or equivalently such that $u \phi_T = \psi_T u$). 

The endomorphism ring $\cE\colonequals \End_k(\phi)$ of $\phi$ consists of the zero map and all isogenies $\phi \to \phi$; it is the centralizer of $\phi(A)$ in $\twist{k}$. 
It is known that $\cE$ is a free finitely generated $A$-module 
with $r\leq \rank_A \cE \leq r^2$. 
We introduce a special element, $\pi=\tau^n$, 
the so-called Frobenius of~$k$. Note that $\pi$ lies in the center of $\twist{k}$, and hence belongs to 
$\cE$.

Isogenies define an equivalence relation on the 
set of isomorphism classes of Drinfeld modules over $k$. 
The isogeny class of $\phi$ is determined by the minimal polynomial of~$\pi$ over $F = \phi(F)$, cf.~\cite[Theorem 4.3.2]{DM}. 
Since the properties of these polynomials are well understood, it is known how to classify Drinfeld modules over finite fields up 
to isogeny.

In this article, we investigate the isomorphism classes within a fixed isogeny class. 
This is an important and difficult question in the theory of Drinfeld modules, which can be approached 
from different viewpoints, cf.~\cite{LaumonCDV, GekelerTAMS}.
Our approach is inspired by the 
work of Waterhouse \cite{Waterhouse} in the case of abelian varieties over finite fields and is 
partly aimed at producing efficient algorithms for explicitly computing a representative of each isomorphism class. 
We refer to Section~\ref{sec:AV} for a more in-depth comparison of our results to known results for abelian varieties.\\ 

When $\cE$ is commutative, the endomorphism ring of a Drinfeld module isogenous to $\phi$ is an $A$-order in $F(\pi)$ 
containing $A[\pi]$. We start by investigating the natural question of when $A[\pi]$ 
itself is an endomorphism ring of a Drinfeld module isogenous to $\phi$. We prove the following: 

\begin{introtheorem}\label{thmA} Let $\phi$ be a Drinfeld module over $k$ such that $\End_k(\phi)$ is commutative. 
	Then $A[\pi]$ is the endomorphism ring of a Drinfeld module isogenous to $\phi$ if and only if 
	either $\phi$ is ordinary or $k=\F_\fp$. 
\end{introtheorem}

Next, we study isogenies from $\phi$ to other Drinfeld modules using the ideals of $\cE$. 
Let $I\trianglelefteq \cE$ be a nonzero ideal. Since $k\{\tau\}$ has a right divison algorithm, we have $k\{\tau\}I = k\{\tau\}u_I$ for some $u_I \in k\{\tau\}$. This element defines an isogeny $u_I\colon \phi \to \psi$, 
where $\psi$ is the Drinfeld module determined by $\psi_T = u_I \phi_T u_I^{-1}$ and is also denoted by $I * \phi$. 
The map $I\mapsto I\ast \phi$ induces a map $S$ from the linear equivalences classes of ideals of $\cE$ 
to the isomorphism classes of Drinfeld modules isogenous to $\phi$. Generally, $S$ is neither injective nor surjective. 

It was observed by Waterhouse \cite{Waterhouse} in the setting of abelian varieties that $S$  
is injective when restricted to ideals of a special type, called \textit{kernel ideals}. 
Kernel ideals were introduced in the context of Drinfeld modules by Yu \cite{JKYu}. 
In Sections~\ref{sec:KerDefs} and~\ref{sec:KerProps}, we revisit Yu's definition, give two other equivalent definitions, and prove 
several general facts about kernel ideals. We also give an explicit example (Example~\ref{ex:nonkernel}) of a rank~$3$ Drinfeld 
module~$\phi$ and an ideal $I\trianglelefteq \End_k(\phi)$ which is not kernel; as far as we know, this is the first 
such explicit example in the published literature. 

In general, we have that $\End_k(I * \phi) \supseteq u_I \mathcal{O}_I u_I^{-1} \cong \mathcal{O}_I$, 
where 
\[
\mathcal{O}_I \colonequals \{ g \in F(\pi) \mid Ig \subseteq I \}.
\]
(Equality holds when $I$ is a kernel ideal; cf.~Lemma~\ref{lem:endker}.) Note that  $\mathcal{O}_I$ 
is an overorder of $\cE$, so $S$ can be surjective only when $\cE$ is the smallest order 
among the endomorphism rings of Drinfeld modules isogenous to $\phi$. 
When $\cE$ is a Gorenstein ring,  we prove that any isogeny $\phi \to \psi$ such that $\mathrm{End}_k(\psi) \cong \mathcal{O}_I$ for some (necessarily kernel) ideal $I \trianglelefteq \mathcal{E}$ arises from the map $S$ via $I \mapsto I * \phi = \psi$. 
In other words, when $\mathcal{E}$ is Gorenstein, the image of $S$ is the set of isomorphism classes in the isogeny class of $\phi$ whose endomorphism rings are overorders of $\mathcal{E}$. Since $A[\pi]$ is a Gorenstein ring, 
we arrive at the following:

\begin{introtheorem}\label{thmB}
	Assume that either $k=\F_\fp$ or the isogeny class that we consider is ordinary, so that there is 
	a Drinfeld module $\phi$ with $\End_k(\phi)=A[\pi]$. Then the map 
$I\mapsto I\ast \phi$ from the linear equivalences classes of ideals of $A[\pi]$ 
	to the isomorphism classes of Drinfeld modules isogenous to $\phi$ is a bijection. 
\end{introtheorem}

In~\cite{Marseglia}, an algorithm is presented for computing the ideal class monoid of an order in a number field. 
In an ongoing project, we are working on adapting that algorithm to orders in function fields. 
Since for a given ideal $I\trianglelefteq \cE$ computing $I*\phi$ is fairly straightforward, Theorem~\ref{thmB} combined with this algorithm 
will provide an efficient method for computing explicit representatives of isomorphism classes of all 
Drinfeld modules isogenous to $\phi$ such that $\End_k(\phi)\cong A[\pi]$. 
Let us mention that Assong~\cite{Assong1} has recently described a brute-force algorithm to list isomorphism classes, based on a theoretical classification in terms of $j$-invariants and ``fine isomorphy invariants'', and implemented this for certain examples of isogeny classes of Drinfeld modules of rank~$3$. Our methods involve fractional ideals in endomorphism rings rather than invariants and explicit expressions for the coefficients of the Drinfeld module.\\

The outline of the paper is as follows. Section~\ref{sec:locmax} contains our analysis of local maximality of $A[\pi]$ at $\pi$, including 
a key result (Theorem~\ref{thm-main}) 
for the proof of Theorem~\ref{thmA}.
Section~\ref{sec:KerDefs} gives the definitions of kernel ideals and proves their equivalence, and Section~\ref{sec:KerProps} gives properties of kernel ideals and proves that every ideal is a kernel ideal when $\mathcal{E}$ is Gorenstein (Proposition~\ref{prop:Gorker}). Section~\ref{sec:ideals} contain our main results: we find which endomorphism rings can occur in a fixed isogeny class (Proposition~\ref{prop:porism}), study the injectivity and surjectivity of the map $I \mapsto I * \phi$ (Theorem~\ref{thm:idealaction}), and prove when $A[\pi]$ occurs as an endomorphism ring (Corollary~\ref{lem:allEnd}), to obtain Theorem~\ref{thmB} (cf.~Corollary~\ref{cor:end}). Finally, Section~\ref{sec:AV} contains a comparison between the results obtained in this paper and results from the literature (\cite{Waterhouse, DeligneOrdAV, CS1}) on abelian varieties over finite fields.

\section{Local maximality at $\pi$ of the Frobenius order}\label{sec:locmax}

As in the introduction, let $A=\F_q[T]$ and $F=\F_q(T)$, and let $k=\F_{q^n}$ be a finite $A$-field, i.e., a 
field equipped with a homomorphism $\gamma\colon A\to k$. We denote $t=\gamma(T).$
Let $\fp\trianglelefteq A$ be the kernel of $A$. Then  
$\fp$ is a maximal ideal such that $\F_{\fp}\colonequals A/\fp=\F_{q^d}$ is a subfield of $\F_{q^n}$. We call $d$ 
the degree of $\fp$; note that $d$ divides $n$. By slight abuse of notation, we denote the monic generator 
of $\fp$ in $A$ by the same symbol. 

Let   
$
\phi\colon A\to \twist{k}
$
be a Drinfeld module of rank $r$, and let $\pi=\tau^n$.  The results about the endomorphism algebra of $\phi$ 
that we use in this section are well-known and can be found, for example, in \cite{GekelerFDM, LaumonCDV, JKYu}, 
but for the convenience of a single reference and consistency of notation, we will refer to \cite{DM}. 

Let $K\colonequals \F_q(\pi)$ be the fraction field of $\F_q[\pi]\subseteq \twist{k}$ and define 
\[
k(\tau)=\twist{k}\otimes_{\F_q[\pi]} K. 
\]
Then $k(\tau)$ is a central division algebra over $K$ of dimension $n^2$, split at all places of~$K$ except at $(\pi)$ 
and $(1/\pi)$, where its invariants are $1/n$ and $-1/n$, respectively; see \cite[Proposition 4.1.1]{DM}. 
Extend $\phi$ to an embedding $\phi\colon F\to k(\tau)$. Then 
\begin{align*}
\cE\colonequals \End_k(\phi) &=\Cent_{\twist{k}}(\phi(A)), \\ 
D\colonequals \End_k(\phi)\otimes_{\phi(A)} \phi(F) &= \Cent_{k(\tau)}(\phi(F)),
\end{align*}
where $\Cent_R(S)=\{x\in R\mid xs=sx\text{ for all }s\in S\}$ denotes the centralizer of a subset~$S$ of a ring $R$.  
To simplify the notation, we will denote $\phi(A)$ by $A$ and $\phi(F)$ by $F$, with $\phi$ being fixed. 
Let 
\begin{align*}
A' &\colonequals \F_q[\pi],\\ 
\tF &\colonequals F(\pi)=\F_q(\phi_T, \pi).
\end{align*} 
Let $B$ be the integral closure of $A$ in  $\tilde{F}$. 
There is a unique place $\tfp$ in $\tF$ over the place~$(\pi)$ of $K$; see \cite[Theorem 4.1.5]{DM}. 
Let $\tF_{\tfp}\colonequals \tF\otimes_{K} \ls{\F_q}{\pi}$ 
be the completion of~$\tF$ at $\tfp$, and let $B_{\tfp}$ be the ring of integers of $\tF_{\tfp}$.  

\begin{defn}\label{def:locmax} Given an $A$-order $R$ in $B$ containing $\pi$, let 
$$
R_{\tfp} \colonequals R\otimes_{\F_q[\pi]} \dvr{\F_q}{\pi} \subseteq B_{\tfp}. 
$$
We say that $R$ is \textit{locally maximal} at $\pi$ if $R_{\tfp}=B_{\tfp}$; cf. \cite[Definition 3.1]{Angles}. 
\end{defn}	 

\begin{rem}
	Suppose $\cE$ is commutative. Then $\cE$ can be considered as an $A'$-order in $\tF$. 
	It is observed in \cite[p. 164]{JKYu} and  \cite[p. 514]{Angles} that 
	$\cE$ is locally maximal at $\pi$. Therefore, for $A[\pi]$ to be an endomorphism ring 
	of a Drinfeld module isogenous to $\phi$ it is necessary for $A[\pi]$ to be locally maximal at $\pi$. We investigate this condition in 
	this section; later we will show that it is also sufficient; cf. Proposition~\ref{prop:porism}. 
\end{rem}

Let $m(x)$ be the minimal polynomial of $\pi$ over $F$, so that $\tF\cong F[x]/(m(x))$. 
Note that $\pi\in \cE$ is integral over $A$, so the polynomial 
$m(x)$ is monic with coefficients in $A$ and $A[\pi]\cong A[x]/(m(x))$. To analyze the local maximality of $A[\pi]$,  
it will be convenient to change the perspective and express $A[\pi]$ as a quotient of $A'[x]$. 
To do so, consider $T$ and $\pi$ as two independent indeterminates over $\F_q$. 
Then consider $m(\pi)\in \F_q[T, \pi]=A'[T]$ as a polynomial $\widetilde{m}(T)$ in indeterminate $T$ 
with coefficients in $A'$. 

\begin{lem}\label{lem:mtilde}
	We have:
	\begin{enumerate}
	\item $\widetilde{m}(T)$ is irreducible in $K[T]$ and has degree $[\tF:K]$ in $T$. 
	\item The leading coefficient of $\widetilde{m}(T)$ is in $\F_q^\times$. 
	\item Let $\bar{m}(T)\in \F_q[T]$ be the polynomial obtained by reducing the coefficients 
	of $\widetilde{m}(T)$ modulo $\pi$. Then, up to an $\F_q^\times$-multiple, $\bar{m}(T)$ is equal to $\fp^{[\tF:K]/d}$. 
	\end{enumerate}
\end{lem}
\begin{proof}
	By \cite[Theorem 4.2.7]{DM}, the degree in $T$ of $m(0)$ is strictly larger than the degrees of the other 
	coefficients of $m(x)$. Hence, the leading term of $\widetilde{m}(T)$ is the leading term of $m(0)\in A$, so its leading coefficient is 
	in $\F_q^\times$. Moreover, by \cite[Theorem 4.2.2 and Theorem 4.2.7]{DM}, up to an $\F_q^\times$-multiple, $m(0)$ is equal to 
	$\fp^{\frac{[\tF:F]\cdot n}{r\cdot d}}$. Using \cite[(4.1.3)]{DM}, we get 
	\[
	(m(0)) = \fp^{\frac{[\tF:K]}{d}}. 
	\]
	Thus, $\deg_T \widetilde{m}(T) = [\tF:K]$. Since $\tF$ is obtained by adjoining a root of $\widetilde{m}(T)$, 
	the equality  $\deg_T \widetilde{m}(T) = [\tF:K]$ implies that $\widetilde{m}(T)$ is irreducible over $K$. 
	Finally, note that $\bar{m}(T)=m(0)$, which completes the proof of the lemma. 
\end{proof}

\begin{lem}\label{key-lem} The following hold:
	\begin{enumerate}
		\item The ideal $\cM$ of $A[\pi]_{\tfp}$ generated by $\pi$ and $\fp$ is maximal.
		\item $A[\pi]_{\tfp}/\cM\cong \F_\fp$. 
		\item The completion $A_\fp$ is a subring of $A[\pi]_{\tfp}$. 
	\end{enumerate}
\end{lem}
\begin{proof} We have 
	\begin{align*}
	A[\pi]_{\tfp} & = A[\pi]\otimes_{\F_q[\pi]} \dvr{\F_q}{\pi}\\ 
	& \cong \F_q[\pi][T]/(\widetilde{m}(T))\otimes_{\F_q[\pi]} \dvr{\F_q}{\pi} & (\text{by Lemma \ref{lem:mtilde}}) \\ 
	& \cong \dvr{\F_q}{\pi}[T]/(\widetilde{m}(T)). 
	\end{align*}
(Note that $\tF \cong K[T]/(\widetilde{m}(T))$ and, because there is a unique place in $\tF$ over $\pi$, \cite[proof of Theorem 2.8.5]{DM}  
implies that $\widetilde{m}(T)$ remains irreducible over $\ls{\F_q}{\pi}$.) 

The element $\pi$ of $A[\pi]_{\tfp}$ is not a unit. Now 
\[
A[\pi]_{\tfp}/(\pi)\cong \F_q[T]/(\bar{m}(T))\cong A/\fp^{[\tF:K]/d}. 
\]
This shows that $\fp$ is also not invertible in $A[\pi]_{\tfp}$ and 
\[
A[\pi]_{\tfp}/(\pi, \fp)\cong A/\fp. 
\]
This proves (1) and (2). 

Next, note that $A[\pi]_{\tfp}$ is an order in $B_{\tfp}$ (because $A[\pi]$ is an order in $B$). 
Hence $A[\pi]_{\tfp}$ is open and closed with respect to the $\tfp$-adic topology on $B_{\tfp}$. 
In particular, $A[\pi]_{\tfp}$ is complete. Now the topology on $A$ induced by the 
embedding $A\to A[\pi]\to B_{\tfp}$ is the $\fp$-adic topology. Hence $A\hookrightarrow B_{\tfp}$ 
extends to an embedding $A_\fp\hookrightarrow B_{\tfp}$. Since $A[\pi]_{\tfp}$ is complete, the image of $A_\fp$ 
lies in $A[\pi]_{\tfp}$. This proves (3). 
\end{proof}

Let 
\begin{align*}
[\tF_{\tfp}:K_\pi] &=e_K \cdot f_K, \\ 
[\tF_{\tfp}:F_\fp] &= e_F \cdot f_F,
\end{align*}
where $F_\fp$ (resp. $K_\pi$) denotes the completion of $F$ (resp. $K$) at $\fp$ (resp. $(\pi)$), 
and where~$e$ and $f$ denote the ramification index and the residue degree of the corresponding extension, respectively. 

\begin{prop}\label{propLocMax} $A[\pi]$ is locally maximal at $\pi$ if and only if one of the following holds: 
	\begin{itemize}
		\item $f_F=1$ and $e_F=1;$
		\item $f_F=1$ and $e_K=1$. 
	\end{itemize}
\end{prop}
\begin{proof}
	Let $\ord_{\tfp}$ be the normalized valuation on $\tF$ corresponding to the place $\tfp$. Then 
	\begin{align*}
	e_K=\ord_{\tfp}(\pi), \\ 
	e_F=\ord_{\tfp}(\fp). 
	\end{align*}

Suppose that $A[\pi]_{\tfp}=B_{\tfp}$. Then, using the notation of Lemma \ref{key-lem}, we have 
$\cM=\tfp$ and $\F_{\tfp} \colonequals B_{\tfp}/\tfp=A[\pi]_{\tfp}/\cM=\F_\fp$. Since $\pi$ and $\fp$ 
generate $\cM$, at least one of them must have $\ord_{\tfp}$ equal to $1$. Hence, either $e_K=1$ or $e_F=1$.  
Next, $f_F$, by definition,  is the degree 
of the extension $\F_{\tfp}/\F_\fp$. 
Hence $f_K=1$. 

Conversely, suppose that one of the given conditions holds. Then the residue field of $B_{\tfp}$ is $\F_\fp$ 
and either $\fp$ or $\pi$ is a uniformizer of $B_{\tfp}$. By the structure theorem of local fields of 
positive characteristic, we have $B_{\tfp}=\dvr{\F_\fp}{\fp}$ or $B_{\tfp}=\dvr{\F_\fp}{\pi}$. 
On the other hand, by Lemma \ref{key-lem}, 
\[
\F_\fp\subseteq A_\fp\cong \dvr{\F_\fp}{\fp}\subseteq A[\pi]_{\tfp}, 
\]
and $\fp, \pi\in A[\pi]_{\tfp}$. Hence $B_{\tfp}\subseteq A[\pi]_{\tfp}$, 
which implies that $B_{\tfp}= A[\pi]_{\tfp}$. 
\end{proof}

\begin{thm}\label{thm-main} Let $H$ be the height of $\phi$ (see \cite[Lemma 3.2.11]{DM} for the definition). Then 
	\[
	\left\lceil \frac{n}{H\cdot d}\right\rceil \leq \frac{[\tF:K]}{d},  
	\]
	with equality if and only if $A[\pi]$ is locally maximal at $\pi$. 
\end{thm}
\begin{proof}
	We have the following equalities:
	\begin{align*}
	\frac{[\tF:F]}{[\tF:K]} &=\frac{r}{n} & (\text{see \cite[(4.1.3)]{DM}}),\\ 
	[\tF:K] & =e_K\cdot f_K  & (\text{see \cite[(4.1.4)]{DM}}),\\ 
	f_K &=f_F\cdot d & (\text{see \cite[(4.1.6)]{DM}}).
	\end{align*}
	On the other hand, by \cite[Proposition 4.1.10]{DM},   
	\[
	H = \frac{r}{[\tF:F]}[\tF_{\tfp}:F_\fp]. 
	\]
	Hence 
	\[
	H=\frac{n}{[\tF:K]}e_Ff_F=\frac{n e_F f_F}{e_Kf_K}=\frac{n}{d} \frac{e_F}{e_K}. 
	\]
	This implies that 
	\[
	\frac{n}{H\cdot d}=\frac{e_K}{e_F}. 
	\]
	On the other hand, 
	\[
	\frac{[\tF:K]}{d}=\frac{e_Kf_K}{d}=e_K\cdot f_F. 
	\]
	Thus, the inequality of the theorem is equivalent to 
	\[
	\left\lceil \frac{e_K}{e_F}\right\rceil \leq e_K\cdot f_F. 
	\]
	Since $e_K, e_F, f_F$ are positive integers, the above inequality always holds, 
	with equality if and only if 
	$f_F=1$ and either $e_K=1$ or $e_F=1$. Now the theorem follows from Proposition \ref{propLocMax}. 
\end{proof}

\begin{rem}
The advantage of having the inequality of Theorem \ref{thm-main}, rather than the statement 
		of Proposition \ref{propLocMax}, is that instead of computing each of $e_K, e_F, f_F$ individually 
		it combines these numbers  into quantities that are easier to compute. 
\end{rem}

\begin{cor}\label{cor:ord}
	If $H\leq r/[\tF:F]$, then $A[\pi]$ is locally maximal at $\pi$. In particular, if~$\phi$ is ordinary, i.e., $H=1$,  
	then $A[\pi]$ is locally maximal at $\pi$. 
\end{cor}
\begin{proof}
	Since $r/[\tF:F]=n/[\tF:K]$, the assumption is equivalent to $n/H\geq [\tF:K]$, which 
	implies that equality in Theorem \ref{thm-main} holds. 
\end{proof}

\begin{cor}\label{cor:Fp}
		If $k=\F_\fp$, i.e., $d=n$, then $A[\pi]$ is locally maximal at $\pi$.
\end{cor}
\begin{proof}
	If $k=\F_\fp$, then $[\tF:F]=r$; see \cite[Proposition 2.1]{GP}. But $[\tF:F]=r$ is equivalent to $[\tF:K]=n$, so $[\tF:K]=d$. 
	Therefore, the inequality of 
	Theorem \ref{thm-main} becomes 
	\[
	1=\left\lceil \frac{1}{H}\right\rceil \leq \frac{[\tF:K]}{d}=1,
	\]
	so it is an equality. 
\end{proof}

\begin{cor}\label{cor:locmax}
	Assume $\End_k(\phi)$ is commutative. Then $A[\pi]$ is locally maximal at $\pi$ 
	if and only if either $\phi$ is ordinary or $k=\F_\fp$. 
\end{cor}
\begin{proof} By \cite[Theorem 4.1.5]{DM}, $\End_k(\phi)$ is commutative if and only if $[\tF:F]=r$. 
	Now, as in the previous proof, $[\tF:K]=n$. The inequality of 
	Theorem \ref{thm-main} becomes 
	\[
	\left\lceil \frac{n}{Hd}\right\rceil \leq \frac{n}{d}. 
	\]
	Since $n/d$ is a positive integer, an equality holds if and only if either $H=1$ or $n=d$. 
\end{proof}

\begin{example}
	Let $\fp=T$, $r=2$, and $n=3$. Let $\phi_T=\tau^2$, so $\phi$ is supersingular. 
	In this case, the characteristic polynomial of the Frobenius is $x^2-T^3$ (since $\pi^2=\tau^6=\phi_T^3$), so $[\tF:F]=2$. 
	Thus, $[\tF:K]=3$. Since $H=2$ and $d=1$, the inequality of Theorem~\ref{thm-main} becomes strict 
	\[
	2=\left\lceil \frac{3}{2\cdot 1}\right\rceil < \frac{3}{1}=3. 
	\]
	Thus, $A[\pi]$ is \textit{not} maximal at $\pi$. 
	One can also see that $A[\pi]$ is not maximal at $\pi$ by directly computing $A[\pi]_{\tfp}$. 
	Indeed, $A[\pi]=A[T\sqrt{T}]$ and $B=A[\sqrt{T}]$. 
	Since $\sqrt{T}$ is the unique prime over $T$, $A[\pi]_{\tfp}=A_T[T\sqrt{T}]\neq A_T[\sqrt{T}]=B_{\tfp}$. 
	Also, note that $\End_k(\phi)= \F_q[\tau]\cong A[\sqrt{T}]$ is the maximal $A$-order in $\tF$. 
\end{example}

\begin{example}
	Suppose $n=6$ and $d=2$. Note that $[\tF:K]$ is divisible by $d$ and divides $n$, 
	(since $r/[\tF:F]=n/[\tF:K]$ and $[\tF:F]$ divides $r$). Hence, $[\tF:K]=2$ or~$6$. The  
	inequality of the theorem becomes 
	\[
	\left\lceil \frac{3}{H}\right\rceil \leq \frac{[\tF:K]}{2}. 
	\]
	Hence $A[\pi]$ is locally maximal at $\pi$ if and only if either $H=1$ or $[\tF:K]=2$. 
	
	For example, when $q=3$, $\fp=T^2+T+2$, $\phi_T=t+\tau^4$, we calculate 
	that $\phi_{\fp}=(2t+1)\tau^2+\tau^8$, which tells us that $H=2$. 
	We also calculate the minimal polynomial for $T$ over $K$, which is given by 
	$\widetilde{m}(x)=x^6+(\pi^2+1)x^3+(\pi^4-\pi^2+2)$. Hence, $[\tF:K]=6$, so $A[\pi]$ 
	is not locally maximal at $\pi$. 
\end{example}

\begin{example}
	Suppose $q=3$, $n=8$, and $\fp=T^2+T+2$. Let
	\[
	\phi_T=t+\tau+(2t+1)\tau^2+2\tau^3+\tau^4.
	\]
	Then, $H=2$ and $\widetilde{m}(x)=x^4+2x^3+2x^2+(2\pi+1)x+\pi^2+\pi+1$, so $[\tF:K]=4$. 
	Thus, 
	\[
	\frac{n}{Hd}=2=\frac{[\tF:K]}{d},
	\]
	so equality in Theorem \ref{thm-main} holds. In this case, $A[\pi]$ is maximal at $\pi$. 
\end{example}

The next four examples show that the quantities in Proposition \ref{propLocMax} are essentially independent of 
each other. 

\begin{example}[Local maximality despite $e_K\neq 1$] 
	Let $q=3$, $\fp=T^2+T+2$, and $k=\F_{q^4}$. Let $\phi_T=t+\tau^2$. By computation, we see 
	that $H=1$ and the minimal polynomial for $T$ over $K$ is given by 
	$\widetilde{m}_T(x)=x^4-x^3+(\pi+2)x^2+(\pi+1)x+\pi^2+1$. In particular, $n/(Hd)=2$ 
	and $[\tF:K]/d=2$. Therefore, $A[\pi]$ is locally maximal at~$\pi$. 
	
	Notice that $e_K/e_F=2$ and $e_Kf_F=2$ imply  that $e_K=2$, $e_F=1$, $f_F=1$, and $f_K=2$. 
\end{example}

\begin{example}[Local maximality despite $e_F\neq 1$] 
Let $q=3$, $\fp=T^2+T+2$, and $k=\F_{q^4}$. Let $\phi_T=t+(t+1)\tau+(t+2)\tau^2+\tau^3$. 
By computation, we see that $H=3$ and the minimal polynomial for $T$ over $K$ 
is given by $\widetilde{m}_T(x)=x^2+x+2\pi^3+2$. In particular, $n/(Hd)=1/3$ and $[\tF:K]/d=1$. 
Therefore, $A[\pi]$ is locally maximal at~$\pi$. 

	Notice that $e_K/e_F=1/3$ and $e_Kf_F=2$ imply that $e_K=1$, $e_F=3$, $f_F=1$, and $f_K=2$. 
\end{example}

\begin{example}[Not locally maximal despite $f_F= 1$] 
	Let $q=3$, $\fp=T^2+T+2$, and $k=\F_{q^6}$. Let $\phi_T=t+\tau+(2t+1)\tau^2$. 
	By computation, we see that $H=2$ and the minimal polynomial for $T$ over $K$ 
	is given by $\widetilde{m}_T(x)=x^6+x^3+\pi^2+2$. In particular, $n/(Hd)=3/2$ and $[\tF:K]/d=3$. 
	Therefore, $A[\pi]$ is not locally maximal at $\pi$. 
	
	Notice that $e_K/e_F=3/2$ and $e_Kf_F=3$ imply that $e_K=3$, $e_F=2$, $f_F=1$, and $f_K=2$. 
\end{example}

\begin{example}[Not locally maximal despite $e_F=e_K= 1$] 
	Assume $d$ is odd, $q$ is odd,  and $n/d=[k:\F_\fp]=2.$ Then there is a supersingular Drinfeld module 
	of rank~$2$ over~$k$ whose minimal polynomial is $x^2+c\fp+c'\fp^2$, where $c, c'\in \F_q^\times$ are such that 
	$c^2-4c'$ is not a square in $\F_q^\times$; see \cite[Example 4.3.6]{DM}. 
	In this case, $\fp$ remains inert in~$\tF$, so $e_F=1$ and $f_F=2$. 
	Since $n/Hd=e_K/e_F$ and $H=2$, we see that $e_K=1$. 
\end{example}


\section{Kernel ideals: Definitions}\label{sec:KerDefs}

We keep the notation of the previous section but from now on we assume that $\cE= \End_k(\phi)$ is commutative. 

Let $I\trianglelefteq \cE$ be a nonzero ideal. Let $\twist{k}I$ be the left ideal of $\twist{k}$ generated by the elements of $I$. 
Then  $\twist{k}I$ is generated by a single element $u_I\in \twist{k}$ since $\twist{k}$ 
has right division algorithm. Thus, $\twist{k}I = \twist{k} u_I$. It follows that 
\[
\twist{k}u_I\phi(A) = \twist{k}I\phi(A) = \twist{k}I = \twist{k} u_I. 
\]
Therefore, $u_I\phi(A)u_I^{-1}\subseteq \twist{k}$. If we set $\psi_T=u_I\phi_Tu_I^{-1}$, then 
$\psi$ is a Drinfeld module over $k$ of rank $r$ and $u_I \colon \phi\to \psi$ is an isogeny. We denote 
$\psi=I\ast \phi$. 

Let $D=\cE\otimes_A F$ be the division algebra of $\phi$. Note that $\cE=D\cap \twist{k}$. 
Hence 
\[
\twist{k}I\cap D\subseteq \twist{k}\cap D=\cE. 
\]
This implies that 
\begin{equation}\label{eq:kernelDE}
\twist{k}I\cap D = (\twist{k}I\cap D)\cap\cE = \twist{k}I\cap (D\cap \cE) = \twist{k}I\cap \cE. 
\end{equation}

\begin{defn}\label{defKI1}
We say that $I$ is a \textit{kernel ideal} if 
$(\twist{k}I)\cap D = I$. This definition is the one in \cite[p. 167]{JKYu}. 
\end{defn}

Next, define 
	\[
	\phi[I]=\bigcap_{\alpha\in I} \ker(\alpha),
	\]
	where $\ker(\alpha)$ denotes the kernel (as a group-scheme) of the twisted polynomial $\alpha\in \twist{k}$ 
	acting on the additive group-scheme $\gm_{a, k}$. 

\begin{lem} We have 
	$\phi[I]=\ker(u_I)$. 
\end{lem}
\begin{proof} Suppose $\alpha\in I$. Then $\alpha\in \twist{k}I$, so $\alpha=f u_I$. Thus $\ker(u_I)\subseteq \ker(\alpha)$, 
	and consequently $\ker(u_I)\subseteq \phi[I]$. Conversely, we can write 
	\[
	u_I=f_1\alpha_1+\cdots+f_m\alpha_m, 
	\]
	for suitable $f_1, \dots, f_m\in \twist{k}$ and $\alpha_1, \dots, \alpha_m\in I$.  This implies that $\phi[I]\subseteq \ker(u_I)$.  
\end{proof}
	
Each $\ker(\alpha)$, $\alpha\in I$, is an $\cE$-module scheme, so $\phi[I]$ is an $\cE$-module scheme. The annihilator 
$\Ann_{\cE}(\phi[I])$ of this module scheme is an ideal of $\cE$. It follows immediately from the definition that $I\subseteq \Ann_{\cE}(\phi[I])$. 

\begin{defn}\label{defKI2}
	We say that $I$ is a \textit{kernel ideal} if $I=\Ann_{\cE}(\phi[I])$. This definition is the analogue 
	of the definition of this concept in the setting of abelian varieties; see \cite[p. 533]{Waterhouse}. 
\end{defn}
	
\begin{lem}\label{lem:KI12}
We have $\Ann_{\cE}(\phi[I])=\twist{k}I\cap D$, so 
	Definitions \ref{defKI1} and \ref{defKI2} are equivalent. 
\end{lem}
\begin{proof}
	Let $J\colonequals \twist{k}I\cap D$ and $J'\colonequals \Ann_{\cE}(\phi[I])$. Suppose $u\in J$. 
	Then $u\in \cE$ and $u=wu_I$ for some $w\in \twist{k}$. But $wu_I$ annihilates $\ker(u_I)=\phi[I]$, 
	so $u\in J'$. This implies that $J\subseteq J'$. Conversely, if $u\in J'$, then by Lemma 2.1.1 in \cite{LaumonCDV} 
	we have $u=wu_I$ for some $w\in \twist{k}$. Hence $u\in \twist{k}u_I\cap \cE=J$, so $J'\subseteq J$. 
\end{proof}

Let $\phi$ and $\psi$ be two Drinfeld module over $k$ of rank $r$. 
Let $\fl$ be a prime not equal to $\fp=\chr_A(k)$. 
Let 
$u\colon \phi\to \psi$ be an isogeny. Then $u$ induces a surjective homomorphism 
${^\phi}\bar{k}\overset{u}{\To} {^\psi}\bar{k}$ of $A$-modules 
with finite kernel, where the notation ${^\phi}\bar{k}$ means that the $A$-module structure on $\bar{k}$ is induced from $\phi\colon A \to k\{\tau\}$ and likewise for $\psi$. From this, we get the short exact sequence 
\[
0\To \Hom_{A_\fl}(F_\fl/A_\fl, {^\phi}\bar{k})\To \Hom_{A_\fl}(F_\fl/A_\fl, {^\psi}\bar{k})\To 
\mathrm{Ext}^1_{A_\fl}(F_\fl/A_\fl , \ker(u)_\fl)\To 0, 
\]
where $\ker(u)_\fl$ denotes the $\fl$-primary part of $\ker(u)$ (this is an \'etale group scheme). Note that 
$T_\fl(\phi)\colonequals \Hom_{A_\fl}(F_\fl/A_\fl, {^\phi}\bar{k})$ is the $\fl$-adic Tate module of $\phi$ and that
\[
\mathrm{Ext}^1_{A_\fl}(F_\fl/A_\fl , \ker(u)_\fl)\cong \Hom_{A_\fl}(A_\fl, \ker(u)_\fl)\cong \ker(u)_\fl.
\]
Hence, $u$ induces an injective homomorphism 
\[
u_\fl\colon T_\fl(\phi)\To T_\fl(\psi)
\]
whose cokernel is isomorphic to 
$\ker(u)_\fl$. On the other hand, on $V_\fl(\phi)\colonequals T_\fl(\phi)\otimes_{A_\fl} F_\fl$, $u_\fl$ induces an isomorphism 
$V_\fl(\phi)\overset{\sim}{\To} V_\fl(\psi)$. Pulling back $T_\fl(\psi)\subseteq V_\fl(\psi)$ via $u_\fl^{-1}$ 
we get an $A_\fl$-lattice $u_\fl^{-1}T_\fl(\psi)$ in $V_\fl(\phi)$ which contains $T_\fl(\phi)$ and a 
short exact sequence 
\begin{equation}\label{eqOverLattice}
0\To T_\fl(\phi)\To u_\fl^{-1}T_\fl(\psi)\To \ker(u)_\fl\To 0. 
\end{equation}

Following \cite[(2.3.6)]{LaumonCDV}, we denote 
\begin{equation}\label{eq:Hl}
H_\fl(\phi)= \Hom_{A_\fl}(T_\fl(\phi), A_\fl). 
\end{equation}
Taking the $A_\fl$-duals of \eqref{eqOverLattice}, we obtain 
\[
0\To \Hom_{A_\fl}(u_\fl^{-1}T_\fl(\psi), A_\fl) \To H_\fl(\phi)\To \mathrm{Ext}^1_{A_\fl}(\ker(u)_\fl, A_\fl)\To 0. 
\]
Note that $\mathrm{Ext}^1_{A_\fl}(\ker(u)_\fl, A_\fl)\cong \Hom_{A_\fl}(\ker(u)_\fl, F_\fl/A_\fl)\cong \ker(u)_\fl$. 
Hence to the isogeny~$u$ there corresponds a canonical sublattice of $H_\fl(\phi)$ whose cokernel is isomorphic to 
$\ker(u)_\fl$. 

Now given a nonzero ideal $I\trianglelefteq \cE$, we would like to describe the sublattice of $H_\fl(\phi)$ corresponding to $u_I$. 
Before doing so we recall an elementary result about the duals of intersections of lattices. 

Let $R$ be a PID with field of fractions $K$. Let $V=K^n$. A lattice in $V$ 
is the $R$-span of a basis of $V$, i.e., a lattice is a free $R$-submodule $\La\subseteq V$ of rank $n$ 
such that $\La K=V$. Fix a basis $\{e_1, \dots, e_n\}$ of $V$ and define a symmetric $K$-bilinear 
pairing $\langle\cdot, \cdot\rangle\colon V\times V\to K$ 
by defining $\langle e_i, e_j\rangle=\delta_{ij}$(= Kronecker symbol) and extending it bilinearly to $V\times V$. 
We identify $V^\ast\colonequals \Hom_K(V, K)$ with the linear functionals on~$V$ and 
take $e_i^\ast(v)=\langle e_i, v\rangle$ as a basis of $V^\ast$. For a lattice $\La$ in $V$, the \textit{dual lattice} 
$\La^\ast\subseteq V^\ast$ is the lattice defined by 
\[
\La^\ast=\{f\in V^\ast\mid f(\la)\in R\text{ for all }\la\in \La\}. 
\]
If we identify $V^\ast$ with $V$ by mapping $e_i^\ast\mapsto e_i$ for all $1\leq i\leq n$, then 
\[
\La^\ast=\{v\in V\mid \langle v, \la\rangle\in R\text{ for all }\la\in \La\}. 
\]
Given two lattices $\La_1, \La_2$ in $V$, it is easy to check that 
\[
\La_1+\La_2=\{\la_1+\la_2\mid \la_1\in \La_1, \la_2\in \La_2\}
\]
is a lattice, and so is 
\[
\La_1\cap \La_2=\{\la\mid \la\in \La_1, \la\in \La_2\}.
\]

\begin{lem} We have 
	\[
	(\La_1\cap \La_2)^\ast = \La_1^\ast+\La_2^\ast. 
	\]
\end{lem}
\begin{proof} The proof is omitted since it is fairly straightforward. 
\end{proof}

Now returning to $u_I$, let $\alpha, \beta\in I$ be nonzero elements. 
The overlattice of $T_\fl(\phi)$ corresponding to 
$\ker(\alpha)\cap \ker(\beta)$ is $\alpha^{-1}T_\fl(\phi)\cap \beta^{-1} T_\fl(\phi)$. 
The sublattice of $H_\fl(\phi)$ corresponding to $\ker(\alpha)_\fl$ is $\alpha H_\fl(\phi)$, so 
$(\alpha^{-1}T_\fl(\phi))^\ast = \alpha H_\fl(\phi)$. From the previous lemma, 
we conclude that the sublattice of $H_\fl(\phi)$ corresponding to $\ker(\alpha)\cap \ker(\beta)$ is 
$\alpha H_\fl(\phi) + \beta H_\fl(\phi)$. Thus, the dual of $u_I^{-1} T_\fl(I\ast \phi)$ is $I H_\fl(\phi)$ and we have proved: 

\begin{lem}\label{lem:IuI}
	The sublattice of $H_\fl(\phi)$ corresponding to $\ker(u_I)_\fl$ is $I H_\fl(\phi)$. 
\end{lem}

Let $\cO_k$ be the ring of integers of the unramified extension $F_k$ of $F_\fp$ with residue field~$k$. 
Let $H_\fp(\phi)$ be the Dieudonn\'e module of $\phi$ as defined in \cite[Sec. 2.5]{LaumonCDV}. Recall 
that $H_\fp(\phi)$ is a free $\cO_k$-module of rank $r$ equipped with a $\tau^{\deg(\fp)}$-linear 
map $f_{\phi, \fp}\colon H_\fp(\phi)\to H_\fp(\phi)$ such that  
\[
\begin{cases}
\fp H_\fp(\phi)\subseteq f_{\phi, \fp}(H_\fp(\phi)) \subseteq H_\fp(\phi), \\ 
\dim_{k} (H_\fp(\phi)/f_{\phi, \fp} (H_\fp(\phi)))=1. 
\end{cases}
\]
Let 
\[
\H(\phi)=\prod_{\fl\trianglelefteq A} H_\fl(\phi),
\]
where the product is over all primes of $A$, including $\fp$. According to \cite[Lemma 2.6.2]{LaumonCDV}, 
there is a bijection between the kernels of isogenies $u\colon \phi\to \psi$ and sublattices 
$M=\prod_{\fl\trianglelefteq A} M_\fl\subseteq \H(\phi)$ 
such that $M_\fl = H_\fl(\phi)$ for all but finitely many primes $\fl$ and $M_\fp$ is 
a free $\cO_k$-submodule of rank $r$ of $H_\fp(\phi)$ such that 
\begin{equation}\label{eq:DieudModCondition}
\begin{cases}
\fp M_\fp\subseteq f_{\phi, \fp}(M_\fp) \subseteq M_\fp, \\ 
\dim_{k} (M_\fp/f_{\phi, \fp}(M_\fp))=1. 
\end{cases}
\end{equation}
The quotient $\prod_{\fl\neq \fp}(H_\fl(\phi)/M_\fl)$ defines a unique finite \'etale $k$-subscheme $G^\fp\subseteq \gm_{a, k}$ in $\phi(A)$-modules. 
Similarly, the quotient $\cO_k$-module 
$H_\fp(\phi)/M_\fp$
endowed with the $\tau^{\deg(\fp)}$-linear map induced by $f_{\phi, \fp}$ defines a unique $k$-subscheme 
$G_\fp\subseteq \gm_{a, k}$ in $\phi(A)$-modules. 
The quotient of $\phi$ by $G^\fp\times G_\fp$ is the isogeny corresponding to $M$. 

\begin{prop}\label{prop:sublattice}
	The sublattice of $\H(\phi)$ corresponding to $u_I$ is $I\H(\phi)\colonequals\prod_{\fl} I H_\fl(\phi)$. 
\end{prop}
\begin{proof}
	We already proved this for $\fl\neq \fp$. On the other hand, $H_\fp(\phi)$ is the contravariant 
	Dieudonn\'e module, so $u_I(H_\fp(I\ast\phi))$ is the submodule generated by all $\alpha H_\fp(\phi)$, $\alpha\in I$. 
	Hence $u_I(H_\fp(I\ast\phi))=I H_\fp(\phi)$. 
\end{proof}

\begin{defn}\label{defKI3}
	Let $I$ be a nonzero ideal of $\cE$. We say that $I$ is a \textit{kernel ideal} if for any ideal $J\trianglelefteq \cE$ 
	the inclusion $J\H(\phi) \subseteq I\H(\phi)$ implies $J\subseteq I$. 
\end{defn}

\begin{lem}\label{lem:KI23}
	Definitions \ref{defKI2} and \ref{defKI3} are equivalent. 
\end{lem}
\begin{proof} 
	Note that by the previous discussion, $J\H(\phi)\subseteq I\H(\phi)$ if and only if $\phi[I]\subseteq \phi[J]$. 
	
	Suppose $I$ is a kernel ideal in the sense of Definition \ref{defKI2} and $\phi[I]\subseteq \phi[J]$. 
	Then 
	\[
	J\subseteq \Ann_\cE\phi[J]  \subseteq \Ann_\cE\phi[I] = I. 
	\]
	Hence $I$ is a kernel ideal in the sense of Definition \ref{defKI3}. 
	
	Conversely, suppose that $I$ is a kernel ideal in the sense of Definition \ref{defKI3}. Denote $J=\Ann_\cE\phi[I]$. 
	We have $I\subseteq J$, and we need to show that this is an equality. 
	For any $\alpha\in J$, $\ker(\alpha)$ contains $\phi[I]$, so $\phi[I]\subseteq \phi[J]$. 
	This implies $J\H(\phi)\subseteq I\H(\phi)$. Hence $J\subseteq I$.  
\end{proof}

The next example shows that in general not every ideal of $\cE$ is a kernel ideal.  

\begin{example}\label{ex:nonkernel}
	Let $q=2$, $\fp=T^4+T+1$, and $n=d=4$. Set $\phi_T =t+t^3\tau^2+\tau^3$. The minimal polynomial 
	of $\pi$ is given by 
	\[
	m(x) = x^3 +Tx^2 +x+\fp.
	\]
	We algorithmically compute, cf.~\cite{GP2}, that an $A$-basis for $\cE$ is given by $e_1$, $e_2$, $e_3$, where
	\[
	e_1=1, \qquad e_2=\pi+1, \qquad e_3=\frac{(\pi+1)^2}{T+1}. 
	\]
	We also compute that 
	\begin{align*}
	e_2e_3 &= e_3e_2 = (T +1)^3 +(T +1)e_3, \\ 
	e_2^2 &= (T + 1)e_3,\\ 
	e_3^2 &= (T +1)^3 +(T +1)^2e_2 +(T +1)e_3.
	\end{align*}
	Let $\fl = T + 1$. We observe that an argument similar to the argument in \cite[Example 4.12]{GP2} 
	implies that $\cE_\fl$ is not Gorenstein. 
	
	Consider the ideal $I = (e_2, e_3)$ in $\cE$. An arbitrary element of $I$ is of the following form:
	\[
	\begin{split}
	& (a_1 + a_2e_2 + a_3e_3)e_2 + (b_1 + b_2e_2 + b_3e_3)e_3 \\
	&=(a_3 +b_2 +b_3)(T +1)^3 +(a_1 +b_3(T +1)^2)e_2 +(b_1 +(a_3 +b_2 +b_3)(T +1))e_3,
	\end{split}
	\]
	where $a_i, b_i\in A$. Hence 
	\[
	I = A(T +1)^3 +Ae_2 +Ae_3.
	\]
	In $\twist{k}$, we have 
	\begin{align*}
	e_2 &=1+\tau^4 \\ 
	e_3 &= t^3 +t^2 +t+(t^3 +t^2 +1)\tau^2 +(t^3 +t)\tau^3 +(t^3 +t^2)\tau^4 +\tau^5.
	\end{align*}
	These polynomials satisfy the equation 
	\[
	w = ue_2 + ve_3,
	\]
	where
	\begin{align*}
	w &\colonequals t^3 +t+1+(t^3 +t^2)\tau +(t+1)\tau^2 +\tau^3, \\ 
	u &\colonequals (t^3 +t^2)^2 +(t^3 +t^2)\tau, \\ 
	v &\colonequals t^3 + t^2.
\end{align*}
	We also have 
	\begin{align*}
	\phi_{(T+1)^2} &= (t^2 +1)+t^3\tau^2 +(t^2 +t+1)\tau^3 +\tau^4 +t\tau^5 +\tau^6 \\ 
	& = (t+(t^2 +1)\tau +(t^2 +t)\tau^2 +\tau^3)w. 
	\end{align*}
	Hence, $(T +1)^2 \in \twist{k}w \subseteq \twist{k} I$. But $I\cap A=(T+1)^3A$, so $(T+1)^2\not\in I$. 
	This proves that $I$ is not a kernel ideal. 
\end{example}


\section{Kernel ideals: Properties}\label{sec:KerProps}
We keep the notation and assumptions of the previous section. 
In particular, $\phi$ is a Drinfeld module over $k$ such that $\cE\colonequals \End_k(\phi)$ is commutative, 
and $D\colonequals \cE\otimes_A F$. 

The next lemma is the analogue of  \cite[Theorem 3.11]{Waterhouse}. 

\begin{lem}\label{lem:linequiv}
	Let $I$ and $J$ be nonzero ideals in $\cE$. 
	\begin{enumerate}
		\item If $I=J u$ for some $u\in D$, then $I\ast\phi\cong J\ast \phi$. 
		\item If $I\ast\phi\cong J\ast \phi$ and $I, J$ are kernel ideals, then $I=J u$ for some $u\in D$. 
	\end{enumerate}
\end{lem}

\begin{proof}
    (1)	Let $\twist{k}I=\twist{k} u_I$ and $\twist{k}J=\twist{k} u_J$. By definition, $(I\ast \phi)_T=u_I \phi_T u_I^{-1}$ 
	and $(J\ast \phi)_T=u_J \phi_T u_J^{-1}$. 
	We have 
	\begin{align*} 
		I\ast\phi\cong J\ast \phi \quad & \Longleftrightarrow  \quad c u_I \phi_T u_I^{-1}c^{-1}=u_J\phi_Tu_J^{-1} & \text{for some $c\in k^\times$}\\ 
		& \Longleftrightarrow u_J^{-1}cu_I\in D \\ 
		& \Longleftrightarrow cu_I=u_J u  & \text{for some $u\in D$}. 
	\end{align*}
If $I=J u$, then $u_I=u_J u$, so $I\ast\phi\cong J\ast \phi$. 

(2) Now assume that $I\ast\phi\cong J\ast \phi$, or equivalently $cu_I=u_J u$. Then 
\[
\twist{k}cu_I=\twist{k}u_I=\twist{k}I
\]
and 
\[
\twist{k}u_J u=\twist{k}Ju.
\]
Note that 
$
\twist{k}Ju \cap D = (\twist{k}J\cap D)u$, so if $I$ and $J$ are kernel ideals, then 
\[
Ju=(\twist{k}J\cap D) u = \twist{k}Ju \cap D =\twist{k}I\cap \cE=I. 
\]
\end{proof}

Let 
\begin{equation}\label{eq:order}
\cO_I \colonequals \{g\in D\mid Ig\subseteq I\}
\end{equation}
be the (right) order of $I$ in $D$. The next lemma is the analogue of \cite[Proposition 3.9]{Waterhouse}. 

\begin{lem}\label{lem:endker}
Let $I$ be a nonzero ideal in $\cE$ and write $\twist{k}I=\twist{k} u_I$ with $u_I~\in~\twist{k}$.  
	\begin{enumerate}
		\item We have $u_I \cO_I u_I^{-1}\subseteq \End_k(I\ast\phi)$. 
		\item If $I$ is a kernel ideal, then $u_I \cO_I u_I^{-1}= \End_k(I\ast\phi)$. 
	\end{enumerate} 
\end{lem}

\begin{proof} 
(1)	Let $u\in \cO_I$. By definition, $u\in D$, so it commutes with $\phi_T$ in $k(\tau)$. Therefore,  
	\[
	(u_Iuu_I^{-1})(u_I\phi_Tu_I^{-1})= u_Iu\phi_Tu_I^{-1}=u_I\phi_Tu u_I^{-1} =(u_I\phi_Tu_I^{-1}) (u_Iuu_I^{-1}). 
	\]
	On the other hand, because $u\in \cO_I$, we have 
	\[
	\twist{k} u_I u   = \twist{k}I u \subseteq \twist{k}I =\twist{k} u_I. 
	\]
	Thus, $\twist{k} u_I u  u_I^{-1}\subseteq \twist{k}$, so $u_I u  u_I^{-1}\in \twist{k}$. 
It follows that $(u_Iuu_I^{-1})\in \End_k(I\ast \phi)$. Hence $u_I\cO_Iu_I^{-1}\subseteq \End_k(I\ast\phi)$. 
	
(2)	Now let $w\in \End_k(I\ast \phi)$. Then $w\in \twist{k}$ and 
	$w(u_I\phi_T u_I^{-1})w^{-1}=u_I\phi_Tu_I^{-1}$. This implies that $u_I^{-1}w u_I\in D$. We have 
	\[
	\twist{k}I (u_I^{-1}w u_I) = \twist{k} u_I(u_I^{-1}w u_I)=\twist{k} w u_I\subseteq \twist{k}u_I=\twist{k}I
	\]
	Assume $I$ is a kernel ideal. Then 
	$
	\twist{k}I\cap D=I$
	and 
	\[
	(\twist{k}I (u_I^{-1}w u_I))\cap D = (\twist{k}I \cap D) (u_I^{-1}w u_I) = I (u_I^{-1}w u_I),
	\]
	where the first equality follows from the fact that $u_I^{-1}w u_I\in D$. We see that 
	\[
	I (u_I^{-1}w u_I) \subseteq I,
	\]
	so $u_I^{-1}w u_I\in \cO_I$. This proves that $\End_k(I\ast \phi)\subseteq u_I\cO_I u_I^{-1}$, which combined with the 
	reverse inclusion proved earlier implies that $\End_k(I\ast \phi)= u_I\cO_I u_I^{-1}$. 
\end{proof}

The next lemma is the analogue of \cite[Theorem 3.15]{Waterhouse}. 
\begin{lem}\label{lem:Emax}
	Assume $\cE$ is the maximal $A$-order in $D$. Then every nonzero ideal of $\cE$ is a kernel ideal. 
\end{lem}

\begin{proof}
	First, consider a nonzero principal ideal $\alpha\cE$. We have $\twist{k}\alpha\cE=\twist{k}\alpha$. Suppose $u=g\alpha\in \twist{k}\alpha$  
	and $u\in D$. Then $g=u\alpha^{-1}\in D$ and $g\in \twist{k}$, so $g\in \cE$. Therefore, $u\in \alpha\cE$. This 
	implies that 
	\[
	\alpha \cE\subseteq \twist{k}\alpha \cap D\subseteq \alpha\cE, 
	\]
	so $(\twist{k}(\alpha\cE))\cap D=\alpha\cE$, i.e., $\alpha\cE$ is a kernel ideal. 
	
	Now let $I\trianglelefteq \cE$ be an arbitrary nonzero ideal. Since $\cE$ is maximal, there is an ideal $J\trianglelefteq \cE$ such that 
	$IJ=\alpha\cE$ is principal. We have 
	\[
	(\twist{k} I\cap D)J\subseteq \twist{k} IJ\cap D = IJ,
	\]
	where the last equality follows from the earlier considered case of principal ideals. Now $I'\colonequals \twist{k} I\cap D$ 
	is an ideal of $\cE$, and we have $I'J\subseteq IJ$. Multiplying both sides by $J^{-1}\subseteq D$, we get $I'\subseteq I$. 
	Since $I\subseteq I'$, we have $I'=I$, so $I$ is a kernel ideal. 
\end{proof}

\begin{defn}
	We say that $\cE$ is Gorenstein if $\cE_\fl\colonequals \cE\otimes_A A_\fl$ 
	is a Gorenstein ring for all primes $\fl\trianglelefteq A$, i.e., $\Hom_{A_\fl}(\cE_\fl, A_\fl)$ 
	is a free $A_\fl$-module of rank $1$; cf. \cite{Bass}.  
\end{defn}

Note that the maximal $A$-order in $D$ is Gorenstein, so the next proposition implies Lemma \ref{lem:Emax}. 

\begin{prop}\label{prop:Gorker}
If $\cE$ is Gorenstein then every nonzero ideal of $\cE$ is a kernel ideal. 
\end{prop}
\begin{proof} Let $I$ and $J$ be nonzero ideals of $\cE$ such that $J H_\fl(\phi) \subseteq I H_\fl(\phi)$. 
	Assume $\fl\neq \fp$. Because $\cE_\fl$ is Gorenstein, $T_\fl(\phi)$ is a free $\cE_\fl$-module 
	of rank $1$; cf. \cite[Theorem 4.9]{GP2}. But then, again because $\cE_\fl$ is Gorenstein,  
	$H_\fl(\phi)=\Hom_{A_\fl}(T_\fl(\phi), A_\fl)$ is also a free $\cE_\fl$-module 
	of rank $1$; cf. \cite[Def. 4.8]{GP2}. Hence, the inclusion $J H_\fl(\phi)\subseteq I H_\fl(\phi)$ implies that $J_\fl\subseteq I_\fl$, 
	where $J_\fl\colonequals J\otimes_A A_\fl$ and $I_\fl\colonequals I\otimes_A A_\fl$.  

	At $\fp$ we consider the decomposition
	\begin{equation}\label{eq:Hpdecomp}
	 H_{\fp}(\phi) = H^c_{\fp}(\phi) \oplus H^{\text{{\'e}t}}_{\fp}(\phi)
	\end{equation}
	of the Dieudonn{\'e} module into its connected component $H^c_{\fp}(\phi)$ and maximal {\'e}tale quotient $H^{\text{{\'e}t}}_{\fp}(\phi)$. 
	Let 
	\begin{equation}\label{eqSplittingD}
	D_\fp\colonequals D\otimes_F F_\fp = \bigoplus_{\nu \vert \fp} D_{\nu},
	\end{equation}
	where the sum is over the places of $\tF=D$ lying over $\fp$ and $D_{\nu}$ is the completion of~$D$ 
	at $\nu$. There is a natural isomorphism (cf. \cite[Theorem 2.5.6]{LaumonCDV})
	\[
	\cE_{\fp} \simeq \mathrm{End}(H_{\fp}(\phi)),
	\]
	where $\mathrm{End}(H_{\fp}(\phi))$ denotes the ring of endomorphisms of $H_{\fp}(\phi)$ compatible 
	with the action of the Frobenius $f_{\phi, \fp}$. By \cite[Corollary 2.5.8]{LaumonCDV}, 
	the splitting \eqref{eqSplittingD} induces a compatible splitting 
	$\cE_{\fp} = \mathcal{E}_{\tilde{\fp}} \oplus \cE'_{\fp}$ such that 
	\begin{align}
	\mathcal{E}_{\tilde{\fp}} &\simeq \mathrm{End}(H_{\fp}^c(\phi)), \label{eq:Tateatp1}\\ 
	\mathcal{E}'_{\fp} &\simeq \mathrm{End}(H^{\text{{\'e}t}}_{\fp}(\phi)) \simeq \mathrm{End}_{A_{\fp}[G_k]}(T_{\fp}(\phi)). \label{eq:Tateatp2}
	\end{align}
	Here $\cE_{\tilde{\fp}}$ is the completion of $\cE$ in $B_{\tilde{\fp}}$, and $\mathcal{E}'_{\fp} = \oplus_{j} \mathcal{E}_{j}$ is a direct sum of finitely many local rings corresponding to places $\nu \neq \tilde{\fp}$ lying over $\fp$, and $T_{\fp}(\phi)=\varprojlim \phi[\fp^n](\bar{k})$ denotes the $\fp$-adic Tate module of $\phi$.
	By~\cite[Corollary, p.~164]{JKYu} we have that $\mathcal{E}_{\tilde{\fp}} = B_{\tilde{\fp}}$ is maximal, hence a DVR, which implies that $H_{\fp}^c(\phi)$ is a free $\mathcal{E}_{\tilde{\fp}}$-module. 
	Further, since $\cE_\fp'$ is Gorenstein by assumption, 
	one can apply the argument in the proof of \cite[Theorem 4.9]{GP2} to \eqref{eq:Tateatp2} to 
	conclude that $H^{\text{{\'e}t}}_{\fp}(\phi)$ is a free $\mathcal{E}'_{\fp}$-module. 
	Combining these statements yields that $J H_{\fp}(\phi) \subseteq I H_{\fp}(\phi)$ also implies that $J_\fp\subseteq I_\fp$.
	
	Finally, consider $I_\fl$ as an $A_\fl$-submodule of $D\otimes_F F_\fl$ for any place $\fl$ including $\fp$. Then 
	\[
	J=\bigcap_\fl (D\cap J_\fl) \subseteq \bigcap_\fl (D\cap I_\fl) =I.  
	\]
	Hence $I$ is a kernel ideal by Definition \ref{defKI3}. 
\end{proof}
 

\section{Endomorphism rings and ideal actions}\label{sec:ideals}

We keep the notation and assumptions of the previous section. In particular, 
$\phi$ is a Drinfeld module over $k$ of rank $r$ such that $\cE=\End_k(\phi)$ is commutative. 

Given an $A$-order $R$ in $\tF=D=\cE\otimes_A F$ and a prime $\fl\lhd A$, 
we denote $R_\fl=R\otimes_A A_\fl$. Also, given a prime $\nu$ of $B$, we denote by $B_\nu$ 
the completion of $B$ at $\nu$ and by $R_\nu$ the completion of $R$ in $B_\nu$. 

The following result is modeled on \cite[Porism 4.3]{Waterhouse}.  

\begin{prop}\label{prop:porism} 
Let $R$ be an $A$-order in $D$ containing $\pi$.   
Then there is a Drinfeld module~$\psi$ in the isogeny class of $\phi$ such that $\End_k(\psi)= R$ 
if and only if $R$ is locally maximal at $\pi$.
\end{prop}

\begin{proof} This is proved in \cite[Theorem 1.5]{Assong2}. We present a slightly 
different argument.

If $R$ is the endomorphism ring of a Drinfeld module isogenous to $\phi$, then 
$R$ contains~$\pi$ and is locally maximal at $\pi$ by \cite[Corollary, p.~164]{JKYu}. 

Conversely, assume $R$ is locally maximal at $\pi$. 
It is enough to show that there is a 
Drinfeld module~$\psi$ in the isogeny class of $\phi$ such that $\End_k(\psi)_\fl= R_\fl$ 
for all the primes $\fl$ of $A$.

Pick any Drinfeld module $\phi_0$ in the isogeny class. For any $\fl \neq \fp$, it follows from our assumptions that the rational Tate module $V_{\mathfrak{l}}(\phi_0)$ is free of rank $1$ over $D_{\mathfrak{l}}$. It therefore contains lattices $L$ with any order $\mathcal{O}_L = \{ x \in D : xL \subseteq L \}$ (cf.~\eqref{eq:order}), and, identifying $V_{\mathfrak{l}}(\phi_0) \simeq D_{\mathfrak{l}}$, we see that such a lattice is Galois invariant if and only if its order contains $\pi$.

For any prime $\mathfrak{l} \neq \mathfrak{p}$, we view both $\End_k(\phi_0)_{\mathfrak{l}} = \End_k(\phi_0) \otimes A_{\mathfrak{l}} \simeq \End_{A_\fl[G_k]}(T_{\mathfrak{l}}(\phi_0))$ and~$R_{\mathfrak{l}}$ as lattices in $V_{\mathfrak{l}}(\phi_0)$. Hence, both $\End_k(\phi_0)$ and $R$ are maximal at all but finitely many primes $\mathfrak{l}$. In particular, there exist only finitely many primes, $\mathfrak{l}_1, \ldots, \mathfrak{l}_n$ say, at which $\End_k(\phi_0)_{\mathfrak{l}} \neq R_{\mathfrak{l}}$. 

The lattice $R_{\mathfrak{l}_1}$ has order $\{ x \in D: x R_{\mathfrak{l}_1} \subseteq R_{\mathfrak{l}_1} \}=R_{\mathfrak{l}_1}$, and so does its dual $R_{\mathfrak{l}_1}^* \simeq R_{\mathfrak{l}_1}$. As in \eqref{eq:Hl}, let $H_{\mathfrak{l}_1}(\phi_0)$ denote the dual of $T_{\mathfrak{l}_1}(\phi_0)$ and consider the intersection $H_{\mathfrak{l}_1}(\phi_0) \cap R^*_{\mathfrak{l}_1}$. This is an order contained in $R^*_{\mathfrak{l}_1}$ and we consider the index 
$\chi:= \chi(R^*_{\mathfrak{l}_1}/(H_{\mathfrak{l}_1}(\phi_0) \cap R^*_{\mathfrak{l}_1}))$, which is a product of non-zero $A$-ideals.
We have 
\[\chi \cdot R^*_{\mathfrak{l}_1} \subseteq H_{\mathfrak{l}_1}(\phi_0) \cap R^*_{\mathfrak{l}_1} \subseteq H_{\mathfrak{l}_1}(\phi_0)
\]
by definition.
So we have obtained an integral lattice $L_{\mathfrak{l}_1} := \chi \cdot R^*_{\mathfrak{l}_1}$ inside $H_{\mathfrak{l}_1}(\phi_0)$, or equivalently, a lattice in $V_{\mathfrak{l}_1}(\phi_0)$ containing $T_{\mathfrak{l}_1}(\phi_0)$, with order 
\[
\{ x \in D: x \chi R^*_{\mathfrak{l}_1} \subseteq \chi R^*_{\mathfrak{l}_1} \} = \{ x \in D: x R^*_{\mathfrak{l}_1} \subseteq R^*_{\mathfrak{l}_1} \} = R_{\mathfrak{l}}.
\] 
Similar constructions yield sublattices $L_{\mathfrak{l}_i}$ of $H_{\mathfrak{l}_i}(\phi_0)$ for all $i=2, \ldots, n$. At all other~$\mathfrak{l} \neq \fp$ we set $L_{\mathfrak{l}} = H_{\mathfrak{l}}(\phi_0)$.

At $\fp$, write $D_{\fp} = \oplus_{\nu \vert \fp} D_{\nu} = D_{\tilde{\fp}} \oplus \left( \oplus_{\nu \neq \tilde{\fp}} D_{\nu} \right) =: D_{\tilde{\fp}} \oplus D'_{\fp}$, cf. \eqref{eqSplittingD}.
In this case, we have that the rational Dieudonn{\'e} module $H^{\text{{\'e}t}}_{\fp}(\phi_0) \otimes F_{\fp} = \oplus_{\nu \neq \tilde{\fp}} (H_{\fp}(\phi_0) \otimes F_{\fp})_{\nu}$ is free over $D'_{\fp}$, where each summand $(H_{\fp}(\phi_0) \otimes F_{\fp})_{\nu}$ is free over $D_{\nu}$, and therefore contains lattices with any order. Comparing $\mathrm{End}(\phi_0)_{\nu}$ and $R_{\nu}$ at each $\nu \neq \tilde{\fp}$ over $\fp$ as lattices in $D_{\nu}$, and adjusting the former if necessary via an analogous procedure to that in the previous paragraph, yields a sublattice $\oplus_{\nu \neq \tilde{\fp}} L_{\nu}$ of $H^{\text{{\'e}t}}_{\fp}(\phi_0)$. At $\tilde{\fp}$, we set $L_{\tilde{\fp}} = H^c_{\fp}(\phi_0)$.

By the dictionary between sublattices of $\mathbb{H}(\phi_0) = \prod_{\fl \trianglelefteq A} H_{\fl}(\phi_0)$ and isogenies, the quotient of $\prod_{\mathfrak{l} \neq \fp} H_{\mathfrak{l}}(\phi_0) \times H_{\mathfrak{p}}(\phi_0)$ by $\prod_{\mathfrak{l} \neq \fp} L_{\mathfrak{l}} \times \prod_{\nu \vert \fp} L_{\nu}$ yields a finite $A$-invariant subgroup~$G$; cf. \cite[Section 2.6]{LaumonCDV}. The quotient $\phi_0 / G$ in turn yields a Drinfeld module $\psi$ isogenous to $\phi_0$, for which $\End_k(\psi)_{\mathfrak{l}} = R_{\mathfrak{l}}$ at all places $\mathfrak{l} \neq \mathfrak{p}$ of $A$, and  $\End_k(\psi)_{\nu} = R_{\nu}$ at all primes $\nu \vert \fp$ of~$\tilde{F}$ with $\nu \neq \tilde{\fp}$. 
Finally, by \cite[Corollary, p. 164]{JKYu}, $\End_k(\psi)$ is locally maximal at $\pi$, so 
we also have $\End_k(\psi)_{\tfp}=R_{\tfp}$. 
\end{proof}

\begin{cor}\label{lem:allEnd}
The ring $A[\pi]$ is the endomorphism ring of a Drinfeld module isogenous to~$\phi$ if and only if either $\phi$ is ordinary or $k = \mathbb{F}_{\mathfrak{p}}$. 
\end{cor}

\begin{proof}
By Proposition~\ref{prop:porism}, 
$A[\pi]$ is the endomorphism ring of a Drinfeld module in the isogeny class of $\phi$ if and only if 
$A[\pi]$ is locally maximal at $\pi$. 
On the other hand, Corollary~\ref{cor:locmax} states that $A[\pi]$ is locally maximal at $\pi$ if and only if either $\phi$ is ordinary or $k = \F_\fp$. 
\end{proof}

We saw in Section~\ref{sec:KerDefs} that, given a Drinfeld module $\phi$ over $k$ and an ideal $I \trianglelefteq \mathcal{E} = \mathrm{End}_k(\phi)$, we can construct an isogenous Drinfeld module $\psi = I * \phi$, which is determined by $\psi_T = u_I \phi_T u_I^{-1}$ and which satisfies $\mathrm{End}_k(\psi) \supseteq u_I \mathcal{O}_I u_I^{-1} \simeq \mathcal{O}_I \supseteq \mathcal{E}$ by Lemma~\ref{lem:endker}.(2).

\begin{thm}\label{thm:idealaction}
Consider the isogeny class of a Drinfeld module $\phi$ over $k$ with commutative endomorphism algebra.
\begin{enumerate}
    \item The map $I \mapsto I * \phi$ defines an action of the monoid of fractional ideals of $\mathcal{E}$ up to linear equivalence on the set of isomorphism classes of Drinfeld modules in the isogeny class of $\phi$ whose endomorphism ring is the order of an $\cE$-ideal (and hence an overorder of $\mathcal{E}$). 
    \item Upon restricting to kernel ideals, the action is free.
    \item If $\cE$ is Gorenstein, then the action is also transitive on the set of all Drinfeld modules whose endomorphism ring is the order of an $\mathcal{E}$-ideal. In other words, if $\mathcal{E}$ is Gorenstein, then every submodule $M$ of $\H(\phi)$ is of the form $I \H(\phi)$ for some nonzero ideal $I\trianglelefteq \cE$. 
\end{enumerate}
\end{thm}

\begin{proof}
(1) By Lemma~\ref{lem:linequiv}.(1), we may consider the fractional ideals of $\mathcal{E}$ up to linear equivalence. The trivial ideal $I = \mathcal{E}$, considered as a $k\{\tau\}$-ideal, is generated by the trivial element, so $\mathcal{E} * \phi = \phi$ for any $\phi$. For two ideals $I,J$ it follows from the definition and commutativity that $(I \cdot J)*\phi = I * (J * \phi)$. As remarked above, for any ideal~$I$, the Drinfeld modules $\phi$ and $I * \phi$ are isogenous via the generator $u_I$ of $k\{\tau\}I$.
    
\noindent (2) This follows from Lemma~\ref{lem:linequiv}.(2).
   
\noindent (3) The proof is inspired by \cite[Proofs of Theorem~4.5 and Theorem~5.1]{Waterhouse}. Suppose that $\phi$ and $\psi$ are isogenous and that $R := \mathrm{End}_k(\psi)$ is the order of an $\mathcal{E}$-ideal, i.e., $R \simeq \mathcal{O}_I$ for some ideal $I \trianglelefteq \mathcal{E}$. We may write $\psi = \phi/G$ where the finite subgroup scheme $G$ is the kernel of the isogeny. We want to show that $\psi \cong I * \phi$. Since $I$ is a kernel ideal by Proposition~\ref{prop:Gorker}, by Proposition~\ref{prop:sublattice} this amounts to showing that the sublattice corresponding to the isogeny $\phi \to \psi$ with kernel $G$ is $I \mathbb{H}(\phi)$, up to linear equivalence.
    (Note also that the Drinfeld module $I * \phi$ indeed has endomorphism ring $u_I \mathcal{O}_I u_I^{-1} \simeq \mathcal{O}_I$ by Lemma~\ref{lem:endker}.(2).)
    
    By the dictionary between lattices and isogenies given above, the kernel~$G$ gives rise to a sublattice of $N_{\mathfrak{l}} \subseteq H_{\mathfrak{l}}(\phi)$ such that $H_{\mathfrak{l}}(\phi)/N_{\mathfrak{l}} 
    \simeq G_{\mathfrak{l}}$ for each $\mathfrak{l} \neq \mathfrak{p}$ and a sublattice $N_{\mathfrak{p}} \subseteq H_{\mathfrak{p}}(\phi)$ satisfying \eqref{eq:DieudModCondition} such that $H_{\mathfrak{p}}(\phi)/N_{\mathfrak{p}} 
    \simeq G_{\mathfrak{p}}$.
    The lattice $N_{\mathfrak{p}}$ in $H_{\mathfrak{p}}(\phi)$ is both a free left $\mathcal{O}_k$-module and a right $\mathcal{E}_{\mathfrak{p}}$-module; by the splittings of $\mathcal{E}_{\mathfrak{p}} = \mathcal{E}_{\tilde{\fp}} \oplus \mathcal{E}'_{\fp}$ and $H_{\fp}(\phi) = H^c_{\fp}(\phi) \oplus H^{\text{{\'e}t}}_{\fp}(\phi)$ in \eqref{eq:Hpdecomp} and their compatibility in~\eqref{eq:Tateatp1} and~\eqref{eq:Tateatp2}, we must have that $N_{\fp} = N_{\tilde{\fp}} \oplus N'_{\fp}$ splits as well, where $N_{\tilde{\fp}}$ is a sublattice of $H^c_{\fp}(\phi)$ and an $\mathcal{O}_k \otimes \mathcal{E}_{\tilde{\fp}}$-module, and $N'_{\fp}$ is a sublattice of $H^{\text{{\'e}t}}_{\fp}(\phi)$ and an $\mathcal{O}_k \otimes \mathcal{E}'_{\fp}$-module.
    
    As remarked in the proof of Proposition~\ref{prop:Gorker}, the Gorenstein property implies that $H_{\mathfrak{l}}(\phi)$ is free over $\mathcal{E}_{\mathfrak{l}}$ of rank~$1$ for all $\fl \neq \fp$, and $H^{\text{{\'e}t}}_{\fp}(\phi)$ is free over $\mathcal{E}'_{\fp}$.
    Hence, any sublattice of $H_{\mathfrak{l}}(\phi)$ is of the form $I_{\mathfrak{l}} \cdot H_{\mathfrak{l}}(\phi)$ for some local ideal $I_{\mathfrak{l}} \trianglelefteq \mathcal{E}_{\mathfrak{l}}$, and any sublattice of $H^{\text{{\'e}t}}_{\fp}(\phi)$ is of the form $I'_{\fp} \cdot H^{\text{{\'e}t}}_{\fp}(\phi)$ for some ideal $I'_{\fp} \trianglelefteq \mathcal{E}'_{\fp}$. Since $G$ is finite, we know that $I_{\mathfrak{l}} = \mathcal{E}_{\mathfrak{l}}$ for all but finitely many~$\mathfrak{l}$; note also that there are only finitely many $\nu \neq \tilde{\fp}$ over $\fp$ that contribute to $I'_{\fp}$.
    Recall that $\mathcal{E}_{\tilde{\mathfrak{p}}}$ is maximal by \cite[Corollary, p.~164]{JKYu}, hence a PID, so again any sublattice of $H^c_{\fp}(\phi) = (H_{\fp}(\phi))_{\tilde{\fp}}$ is of the form $I_{\tilde{\fp}} \cdot H^c_{\fp}(\phi)$ for a local principal ideal $I_{\tilde{\fp}} \trianglelefteq \cE_{\tilde{\fp}}$. 
    Note that at all places, we may scale the ideal generators to lie in the local endomorphism ring.
    We conclude that
    \[
    N_{\fp} = \left( I_{\tilde{\fp}} \cdot H^c_{\fp}(\phi) \right) \oplus \left( I'_{\fp} \cdot H^{\text{{\'e}t}}_{\fp}(\phi) \right) = I_{\fp} \cdot H_{\fp}(\phi)
    \]
    for some local ideal $I_{\mathfrak{p}} = I_{\tilde{\fp}} \oplus I'_{\fp}$ of~$\mathcal{E}_{\mathfrak{p}}$.
    
    These local ideals $I_{\fp}$ and $I_{\fl}$ for all $\fl \neq \fp$, i.e., local integral lattices, are the localizations of a global lattice (again since $I_{\mathfrak{l}} = \mathcal{E}_{\mathfrak{l}}$ for all but finitely many $\mathfrak{l}$), which is closed under the action of $\mathcal{E}$ since it is so everywhere locally by construction. Hence, it is a global ideal $I$, as we had to show.
\end{proof}

\begin{cor}\label{cor:end}
Suppose that $\cE= A[\pi]$ (so that either $\phi$ is ordinary or $k = \mathbb{F}_{\mathfrak{p}}$, by Lemma~\ref{lem:allEnd}). Then the action $I \mapsto I * \phi$ of the monoid of fractional ideals of $A[\pi]$ is free and transitive on the isomorphism classes in the isogeny class of $\phi$.
\end{cor}

\begin{proof}
Since we consider the fractional ideals up to linear equivalence, we may without loss of generality consider only integral $A[\pi]$-ideals. Since $A[\pi]$ is Gorenstein (cf. \cite[Proposition 4.10]{GP2}), every $A[\pi]$-ideal is a kernel ideal by Proposition~\ref{prop:Gorker}. The statement now follows from Theorem~\ref{thm:idealaction} since every endomorphism ring is an overorder of $A[\pi]$; note that all such overorders occur as endomorphism rings by Proposition~\ref{prop:porism}.
\end{proof}

\begin{rem}
\begin{enumerate}
    \item The ideal action already appears in \cite[Section 3]{HayesCFT} in a slightly different setting: fix an $A$-order $\mathcal{O}$ and consider the Picard group $\mathrm{Pic}(\mathcal{O})$, i.e., the quotient group of invertible $\mathcal{O}$-ideals modulo principal ideals. Hayes shows that $\mathrm{Pic}(\mathcal{O})$ acts on the isomorphism classes of Drinfeld modules whose endomorphism ring contains~$\mathcal{O}$. Invertible $\mathcal{O}$-ideals are proper and therefore have order $\mathcal{O}$; so this statement is consistent with the statement $\mathrm{End}_k(I * \phi) \supseteq u_I \mathcal{O}_I u_I^{-1} \cong \mathcal{O}_I$ which we prove in Lemma~\ref{lem:endker}.
    \item It follows from Theorem~\ref{thm:idealaction} that the number of isomorphism classes in the isogeny class is bounded below by the sum of the class numbers of the overorders of $\mathcal{E}$ and that equality holds if $\mathcal{E}$ is minimal and Bass, so that every overorder is Gorenstein. For rank~$2$ Drinfeld modules, this result can also be found in \cite[\S 6]{GekelerTAMS}  where the class numbers are given as products involving Dirichlet characters. In higher rank, analogous expressions for the class numbers of the orders in $D$ could be given.
    \item The problem of describing endomorphism rings of Drinfeld modules has been considered by several authors, see e.g.~\cite{Angles, GekelerFDM, JKYu}. Explicit algorithms to compute endomorphism rings were developed in \cite{GP} for Drinfeld modules of rank $2$ over $\mathbb{F}_{\mathfrak{p}}$ and in \cite{GP2} for any Drinfeld module with commutative endomorphism algebra; in \cite{KuhnPink}, the existence of an effective general algorithm is shown.
\end{enumerate}

\end{rem}


\section{Comparison with abelian varieties over finite fields}\label{sec:AV}

There are striking resemblances of the theory of Drinfeld modules over finite fields 
with the theory of abelian varieties over finite fields. Isogeny classes of such abelian varieties are also determined by the minimal or characteristic polynomial of their Frobenius endomorphism $\pi$, and it is an important open problem to describe the isomorphism classes within a fixed isogeny class. 
Indeed, precisely when the varieties are ordinary or defined over the prime field $\mathbb{F}_p$, there exist categorical equivalences between isomorphism classes of abelian varieties over $\mathbb{F}_q$ and certain $\mathbb{Z}[\pi, \bar{\pi}]$-ideals, where $\bar{\pi} = q/\pi$ is the dual of the Frobenius, also called the Verschiebung.

First, consider an isogeny class of simple ordinary abelian varieties over $\mathbb{F}_q$ determined by a Frobenius endomorphism $\pi$. It is known that any ordinary variety $A/\mathbb{F}_q$ admits a (Serre-Tate) canonical lifting $\tilde{A}$ to the Witt vectors $W = W(\overline{\mathbb{F}}_q)$, which may be embedded into $\mathbb{C}$. In \cite{DeligneOrdAV}, Deligne shows that the functor $A \mapsto H_1(\tilde{A} \otimes_W \mathbb{C})$ induces an equivalence of categories between isomorphism classes in the isogeny class determined by $\pi$ and free $\mathbb{Z}$-modules of rank $2 \dim(A)$ equipped with an endomorphism $F$ acting as $\pi$ and an endomorphism $V$ such that $FV = q$ playing the role of Verschiebung; these modules are often called Deligne modules. On the other hand, complex abelian varieties $A_{\mathbb{C}}$ are determined by lattices via the equivalence $A_{\mathbb{C}} \mapsto A_{\mathbb{C}}(\mathbb{C}) \cong \mathbb{C}^g/\Lambda$ induced from complex uniformization, and when $A_{\mathbb{C}}$ has CM through a CM-type $\Phi$, we may write $\Lambda = \Phi(I)$ for some fractional $\mathrm{End}(A_{\mathbb{C}})$-ideal~$I$. In this way, we may associate a fractional ideal $I$ to each ordinary abelian variety $A/\mathbb{F}_q$, since each variety over $\mathbb{F}_q$ has CM and therefore so does its canonical lifting~$\tilde{A}$. Linearly equivalent fractional ideals yield homothetic lattices and hence isomorphic abelian varieties, and homomorphisms between abelian varieties are described by quotient ideals.  Put differently, fractional ideals up to linear equivalence act on the isomorphism classes in the ordinary isogeny class.

By comparison, it should follow with a similar proof that ordinary Drinfeld modules over $k$ admit a canonical lifting to $\mathbb{C}_{\infty}$ of $A$-characteristic zero. On the one hand, Drinfeld modules over~$\mathbb{C}_{\infty}$ admit a analytic uniformization by a lattice $\Lambda \subseteq \mathbb{C}_{\infty}$ (where homothetic lattices describe isomorphic Drinfeld modules), which yields a bijection between lattices in $\mathbb{C}_{\infty}$ and Drinfeld modules over $\mathbb{C}_{\infty}$. On the other hand, the ideal action $\phi \mapsto I * \phi$ may be defined for arbitrary Drinfeld modules over any $A$-field (i.e., of any characteristic).
Indeed, as alluded to above, the Picard group of fractional ideals of an order $\mathcal{O}$ up to linear equivalence acts simply transitively on the isomorphism classes of Drinfeld modules over $\mathbb{C}_{\infty}$ with CM by $\mathcal{O}$. 
Ideals of $\mathcal{E}$ may be embedded in $\mathbb{C}_{\infty}$ as lattices, and every lattice $\Lambda \subseteq \mathbb{C}_{\infty}$ yields an ideal $\chi(\Lambda/\mathcal{E})\Lambda \trianglelefteq \mathcal{E}$.
One can show that if $u_I: \phi \to I * \phi$ is an isogeny and $\phi$ corresponds to the lattice~$\Lambda$, then $I*\phi$ corresponds to the lattice $I^{-1} \Lambda$, where $I^{-1} = (\mathcal{E}:I)$.
This shows that, like for abelian varieties, the ideal action for ordinary Drinfeld modules can equivalently be described in terms of lattices via analytic uniformization on the lifted modules.

Next, consider an isogeny class of simple abelian varieties over $\mathbb{F}_p$, determined by a characteristic polynomial of $\pi$ which does not have real roots, to ensure the endomorphism rings are all commutative. In \cite{CS1}, the authors show that such an isogeny class contains an element $A_w$ with minimal endomorphism ring, i.e., $\mathrm{End}_{\mathbb{F}_p}(A_w) = \mathbb{Z}[\pi,\bar{\pi}]$, which is Gorenstein. They then use this variety to show that the functor $A \mapsto \mathrm{Hom}(A,A_w)$ induces a contravariant equivalence between isomorphism classes in the isogeny class and reflexive $\mathbb{Z}[\pi,\bar{\pi}]$-modules, which are in turn equivalent to finite free $\mathbb{Z}$-modules with an endomorphism $F$ acting as $\pi$ and an endomorphism $V$ such that $FV = p$ which plays the role of $\bar{\pi}$. When the varieties are also ordinary, the authors also prove that their functor is equivalent to that of Deligne.

By comparison, for Drinfeld modules over $k = \mathbb{F}_{\mathfrak{p}}$, the existence of an isomorphism class~$\phi_w$ with minimal endomorphism ring $\mathrm{End}_k(\phi_w) = A[\pi]$ is guaranteed by Corollary~\ref{cor:Fp} and Lemma~\ref{lem:allEnd}. The functor $\phi \mapsto \mathrm{Hom}_k(\phi,\phi_w)$ from isomorphism classes in the isogeny class of $\phi_w$ to reflexive $A[\pi]$-modules can be proven to be fully faithful by using Tate's theorems for Drinfeld modules and mimicking the proof of fully faithfulness in \cite[Theorem 25]{CS1}. Essential surjectivity follows from the main result in \cite{LM}, when we view an $A[\pi]$-module as an $A$-matrix with characteristic polynomial determined by that of $\pi$. 
Moreover, suppose that $\phi = I * \phi_w$ for some (necessarily kernel) ideal $I \trianglelefteq \mathcal{E} = A[\pi]$ and recall that $\phi_T = u_I (\phi_w)_T u_I^{-1}$ for $u_I \in k\{\tau\}$ with $k\{\tau\}I = k\{\tau\}u_I$. Then, using that $A[\pi] = \mathrm{End}_{k}(\phi_w) = \{ v \in k\{\tau\} : v (\phi_w)_T = (\phi_w)_T v \}$, we see that 
\[
\begin{split}
    \mathrm{Hom}_k(\phi, \phi_w) & = \{ u \in k\{\tau\} : u \phi_T = (\phi_w)_T u \} \\
    & = \{ u \in k\{\tau\} : u u_I (\phi_w)_T = (\phi_w)_T u u_I \} \\
    & = \{ u \in k\{\tau\} : u u_I \in A[\pi] \} \\
    & = k\{\tau\} u_I \cap A[\pi] = I,
\end{split}
\]
where the last equality follows from the definition of a kernel ideal, see Definition \ref{defKI1} above and cf.~Equation~\eqref{eq:kernelDE}. This shows that the two constructions are in fact equivalent for Drinfeld modules.


 \renewcommand{\bibliofont}{\normalsize}
\providecommand{\bysame}{\leavevmode\hbox to3em{\hrulefill}\thinspace}
\providecommand{\MR}{\relax\ifhmode\unskip\space\fi MR }
\providecommand{\MRhref}[2]{%
  \href{http://www.ams.org/mathscinet-getitem?mr=#1}{#2}
}
\providecommand{\href}[2]{#2}

\end{document}